\documentclass[11pt,a4paper,reqno]{amsart}

\usepackage{times} 
\usepackage{amsmath} 
\usepackage{amssymb}  
\usepackage{amsthm}
\usepackage{latexsym}
\usepackage{amsfonts,bbm}
\usepackage{xcolor}
\usepackage{mathtools}
\usepackage{enumerate}
\usepackage{cite}
\usepackage{tikz}
\usepackage{graphicx}
\usepackage{xcolor}

\usepackage[colorinlistoftodos,prependcaption,textsize=small]{todonotes}
\usepackage{xargs}  
\newcommandx{\ms}[2][1=]{\todo[linecolor=red,backgroundcolor=red!25,bordercolor=red,#1]{#2}}
\newcommand{\changed}[1]{{\leavevmode\color{black}#1}}

\newtheorem{thm}{Theorem}[section]
\newtheorem{cor}[thm]{Corollary}
\newtheorem{lem}[thm]{Lemma}
\newtheorem{prop}[thm]{Proposition}
\theoremstyle{definition}
\newtheorem{defn}[thm]{Definition}

\numberwithin{equation}{section}
\allowdisplaybreaks

\newcommand{\R}{\mathbb{R}}

\newcommand{\calC}{\mathcal{C}}
\newcommand{\calD}{\mathcal{D}}

\newcommand{\dom}[1]{D(#1)}
\DeclareMathOperator{\ran}{ran}

\newcommand{\sbvek}[2]{\left[\begin{smallmatrix}#1\\#2\end{smallmatrix}\right]}
\newcommand{\bvek}[2]{\begin{bmatrix}#1\\#2\end{bmatrix}}

{\left(\begin{smallmatrix}}
{\end{smallmatrix}\right)}

\newenvironment{smallbmatrix}
{\left[\begin{smallmatrix}}
{\end{smallmatrix}\right]}

\title[PH structures in optimal control: primal-dual gradient method and CbI]{Port-Hamiltonian structures in infinite-dimensional optimal control: Primal-Dual gradient method and control-by-interconnection}
\author{Hannes Gernandt and Manuel Schaller}

\begin{document}
\begin{abstract}
In this note, we consider port-Hamiltonian structures in numerical optimal control of ordinary differential equations. By introducing a novel class of nonlinear monotone port-Hamiltonian (pH) systems, we show that the primal-dual gradient method may be viewed as an infinite-dimensional nonlinear pH system. The monotonicity and the particular block structure arising in the optimality system is used to prove exponential stability of the dynamics towards its equilibrium, which is a critical point of the first-order optimality conditions. Leveraging the port-based modeling, we propose an optimization-based controller in a suboptimal receding horizon control fashion. To this end, the primal-dual gradient based optimizer-dynamics is coupled to a pH plant dynamics in a power-preserving manner. We show that the resulting model is again monotone pH system and prove that the closed-loop exhibits  local exponential convergence towards the equilibrium.

\end{abstract}
\maketitle

\section{Introduction}
\noindent Physics-based modeling based on the underlying systems energy has received a lot of attention in recent years. In this context, a very prominent class is given by \emph{port-Hamiltonian (pH) systems}. This modeling class extends the well-established Hamiltonian approach to systems with input and outputs, the so-called ports. These ports allow for interaction of pH systems with the environment, e.g., by control, estimation, or by performing structure-preserving coupling. The latter in particular renders pH modeling appealing for complex applications, as it enables modular assembly of complex systems arising in district heating networks~\cite{hauschild2020port}, power networks~\cite{fiaz2013port,GernHins22,GerHRS21}, fluid dynamics~\cite{rashad2021port,Scha_Reis24}, electromagnetics~\cite{clemens2024port,ReisStykel23} or irreversible thermodynamics~\cite{ramirez2022overview}, to name a few.
As pH systems are dissipative by design, they provide a rich framework for mathematical analysis~\cite{jacob2012linear,skrepek2021well, le2005dirac}, numerical methods~\cite{MehU22,kotyczka2018weak}, or control design~\cite{SchJ14,ortega2002interconnection,breiten2023structure}. Port-Hamiltonian structures have also been utilized in optimal control and optimization, i.e., to study optimal network flow problems \cite{DogaKLST2023} or energy-optimal control~\cite{schaller2021control,faulwasser2022optimal}, see also \cite{schaller2024energy} for an application to adaptive high-rise buildings.

In addition to pH systems, which are considered as a model of the underlying physical system, one may also have port-Hamiltonian structures by means of numerical methods applied to the problem. One of the most prominent examples are \emph{gradient-type methods} formulated in continuous time, which often lead to pH systems. In an optimization context, the primal-dual gradient method was formulated in the pH framework in \cite{StegPers15}. There, the inherent dissipative structure and recent generalizations of LaSalle's invariance principle~\cite{CherMall2016} are exploited to show convergence of the primal-dual gradient dynamics towards the set of points satisfying the Karush-Kuhn-Tucker conditions. For non-quadratic but convex cost functions, the formulation of the primal-dual gradient dynamics as a continuous-time system requires the notion of pH systems with a monotone nonlinearity, stemming from the nonlinear derivative of the convex cost function. A first framework for monotone pH systems in finite dimensions was recently proposed in \cite{CamSch23}. There, the authors focus on the geometric description of the system class based on monotone structures rather than on analytic aspects such as the existence of solutions or characterization of stability. 

In optimal control, the optimization variable is separated into a state and a control variable. 
Thus, optimal control is the basis for various control algorithms. A well-known representative of an optimization-based controller is the Riccati feedback for linear systems.
For nonlinear problems with state and control constraints, a widely used optimization-based controller is \emph{Model Predictive Control} (MPC; \cite{Gruene17,RawlMayn17}). After obtaining a state measurement or a state estimation of the system to be controlled (in the following called plant), this state is set as an initial condition, and an optimal control problem on a finite horizon is solved. Then, an initial part of the optimal control on this finite horizon is fed back into the system to be controlled and the process is repeated. This optimization-based feedback controller controller is well-studied for various system classes and in particular provides access to suboptimality and stability guarantees. However, in complex applications, solving the optimal control problem may be infeasible. Thus, one may rely on suboptimal solutions as feedback, leading to \emph{suboptimal MPC} \cite{GeSc_ScokMayn99}. A very successful suboptimal MPC scheme is \emph{real-time iteration}~\cite{GeSc_DiehBock05}. In its original form, one only performs one step of a Lagrange-Newton method for the optimal control problem and feeds back the control of the current iterate into the system. If the time step is small enough compared with the nonlinearity of the problem, it may be shown that the plant and the Newton method converge simultaneously, hence achieving asymptotic stability of the controlled system. In \cite{Zanelli21} a proof based on Lyapunov arguments and allowing for convergent optimization algorithms convergent to $Q$ was provided. Another suboptimal MPC scheme is \emph{instant model predictive control} \cite{GeSc_YoshInou2019}, where one primal-dual gradient step is performed and the resulting control is used as a~suboptimal feedback. Recently, a structured analysis of suboptimal MPC algorithms was presented in \cite{Karapet23}, see also \cite{Haeberle20} for a related problem of iteratively constructing optimal feedback laws for finite-dimensional optimal control problems.

When modeling the dynamics of the optimization algorithm (in the following called \emph{optimizer dynamics}) in optimal control as a pH system, suboptimal MPC schemes may also be formed by coupling the optimizer dynamics to a pH plant in the spirit of \emph{control-by-interconnection (CbI)}~\cite{ortega2002interconnection}. Since both systems exhibit a port-Hamiltonian structure, a structure-preserving controller plant interconnection leads to a pH coupled optimizer-plant dynamics. For finite-dimensional problems, i.e., time-discretized ordinary differential equations, this approach was recently proposed in \cite{Pham22}. There, the convergence result is based on the analysis of primal-dual dynamics for finite-dimensional optimization problems in \cite{StegPers15}. Recently, in \cite{vu2023port}, the approach of \cite{Pham22} was extended to also include a port-Hamiltonian observer to estimate the current state of the plant as an initial condition for the optimal control problem. In the works mentioned above, all considered systems are finite dimensional, i.e., both the plant and the optimizer dynamics. In this work, we provide an approach allowing for an infinite-dimensional optimizer dynamics. This case occurs e.g. when considering optimal control problems with time-continuous ordinary differential equations or partial differential equations. 

The main contributions of this paper are the following. First, to prepare for a pH formulation of the optimizer dynamics, we define a novel class of infinite-dimensional nonlinear monotone port-Hamiltonian systems. 
We provide well-posedness results, a power-balance equation, and closedness under power-preserving coupling. 
Moreover, we prove an exponential stability result particularly tailored to block operators appearing in optimal control leveraging a recent sufficient condition for exponential stability of semigroups generated by dissipative linear block operator matrices from \cite{GernHins22}.  
As a second contribution, we derive infinite-dimensional optimizer dynamics for the primal-dual gradient method and prove that these may be considered as a monontone nonlinear \changed{pH system}. We prove exponential stability of the optimizer dynamics, that is, convergence to the optimal solution. 
Thirdly, we endow the optimizer dynamics with input and output ports and couple them to a pH plant in a power-preserving manner. Remarkably, this coupling utilizes the initial part of the optimal solution, as in MPC, however, does not necessitate the use of an observer and is directly formulated via the plant output. We show that the resulting closed-loop system is again a monotone \changed{pH system} and prove exponential stability. 

The paper is organized as follows. Section~\ref{sec:montone_phs} presents a novel class of monotone pH systems in Hilbert spaces, along with an examination of their properties, including passivity, stability, and structure-preserving interconnections. Section~\ref{sec:optimizer_dynamics} contains a derivation of the optimizer dynamics through the optimality system linked to the optimal control problem. Finally, in Section~\ref{sec:coupled_optimizer_plant}, we interconnect optimizer dynamics with finite-dimensional monotone plant dynamics, demonstrating that the resulting interconnected system forms a~monotone pH system that converges towards a~specified equilibrium.

\medskip

\noindent \textbf{Notation and preliminaries.}
 For a closed and densely defined (possibly nonlinear) operator $A:X \supseteq D(A) \to X$ with domain $D(A)$ on a real Hilbert space $(X,\langle \cdot,\cdot\rangle_X)$ with norm $\|x\|_X = \sqrt{\langle x,x\rangle}$, we denote the \emph{resolvent set} by \[\rho(A):=\{\lambda\in\mathbb{C}~| \text{$\lambda I-A$ is bijective}\}.\] 
Let $(Y,\langle \cdot,\cdot\rangle_Y)$ be another (real) Hilbert space. Then an operator $F:X\supseteq D(F)\rightarrow Y$ is called \emph{Frechét differentiable in $x\in D(F)$} if there exists bounded linear $\mathrm{D}F(x):X\rightarrow Y$ satisfying
\[
F(x+h) = F(x) + \mathrm{D}F(x)h + o(h),\quad \text{for all}\,\, h\in X\,\, \text{with}\,\, x+h\in D(F),
\]
where the remainder term $o:X\rightarrow Y$ satisfies $\tfrac{\|o(z)\|_Y}{\|z\|_X}\to 0$ for $\|z\|_X\to 0$. 

If $F$ is Fréchet differentiable at all $x\in D(F)$, then we simply call $F$ Fréchet differentiable, or just \emph{differentiable}. If $Y=\R$ then $\mathrm{D}F(x)\in X^*$ and by means of the Riesz representation theorem, we can identify $\mathrm{D}F(x)$ with the (unique) gradient $\nabla F(x)\in X$ of $F$ at $x$, satisfying $\mathrm{D}F(x)y=\langle\nabla F(x),y\rangle_X$ for all $y\in X$. 

A function $F:X\supseteq D(F)\rightarrow\mathbb{R}$ is called \emph{convex} if for all $\beta\in[0,1]$ and all $x_1,x_2\in D(F)$, 
 \begin{align*}
 F(\beta x_1+(1-\beta)x_2)\leq \beta F(x_1)+(1-\beta)F(x_2)
 \end{align*}
and if this inequality is strict for all $x_1\neq x_2$, then $F$ is called \emph{strictly convex}. Furthermore, if $F$ is differentiable with gradient $\nabla F$, then, see e.g.~\cite[Proposition 5.5]{EkelTema99}, $F$ is convex, if and only if for all $x,y\in D(F)$
\[
\langle x-y,\nabla F(x)-\nabla F(y)\rangle_X\geq 0
\]
\changed{and strictly convex if and only if there is $\alpha \geq 0$ such that for all $x,y\in D(F)$
\begin{align}\label{def:convex_deriv}
\langle x-y,\nabla F(x)-\nabla F(y)\rangle_X\geq \alpha \|x-y\|^2.
\end{align}}
\changed{For $t_f\geq 0$, we denote the space of continuous functions on $[0,t_f]$ with values in $X$ by $C([0,t_f];X)$ endowed with the usual maximum norm. Further, $W^{1,1}((0,t_f],X)$, denotes the space of weakly differentiable and integrable functions from $[0,t_f]$ to $X$ with integrable weak derivative.}

\section{Monotone Port-Hamiltonian systems}
\label{sec:montone_phs}
\noindent In this section, we propose a class of nonlinear port-Hamiltonian (pH) systems in a Hilbert space where the state-dependent part of the system dynamics is governed by an accretive operator. We prove shifted passivity, study convergence towards steady states and show that the class is closed under power-preserving interconnections.

First, we present the formal definition of the considered class.
\begin{defn}[\changed{(Maximally) monotone pH system}]
\label{def:mono_phs}
    Let $(X,\langle\cdot,\cdot\rangle)$ be a real Hilbert space and let $M:X\supseteq D(M)\rightarrow X$ be an (not necessarily linear) \emph{accretive} operator, that is,
\begin{align}
    \label{eq:M_dissip}
\langle M(x_1)-M(x_2),x_1-x_2 \rangle\geq 0 \quad \text{for all $x_1,x_2\in D(M)$.}
\end{align}
Moreover, let $(U,\langle \cdot,\cdot\rangle_U)$ be another real Hilbert space and $B:U\rightarrow X$ be linear and bounded. Then, a~\emph{monotone pH system} is defined by the equations
\begin{align}\tag{mPHS}
\label{eq:monotone_phs}
\begin{split}
    \tfrac{\mathrm{d}}{\mathrm{d}t}x(t) &= -M(x(t)) + Bu(t),\quad x(0)=x_0,\\
    y(t) &= B^* x(t),
\end{split}
\end{align}
\changed{for all $t\geq 0$, where $u:[0,\infty)\to U$ is an input and $x_0\in X$ is an initial state.}
If the governing operator $M$ is \emph{m-accretive}, that is, in addition to \eqref{eq:M_dissip} the surjectivity condition $\ran(\lambda I+M)=X$ holds for some \changed{(and hence for all) $\lambda>0$, see, e.g., \cite[Proposition 3.3]{Barb10}, then }\eqref{eq:monotone_phs} is called a \emph{maximally monotone pH system}.
\end{defn}
\changed{Note that in the above Definition~\ref{def:mono_phs}, we only introduce the abstract system class. In the following, we thus introduce suitable solution concepts which in view of the nonlinearity requires additional regularity of $u$ and $x_0$.}

\noindent \textbf{Homogeneous problems.} For m-accretive operators in a Hilbert space $X$, we may use nonlinear semigroups to study the  solutions of the nonlinear homogeneous Cauchy problem associated with \eqref{eq:monotone_phs} given by
\begin{align}
\label{eq:mono_cauchy}
 \tfrac{\mathrm{d}}{\mathrm{d}t}x(t) = -M(x(t)),\quad x(0)=x_0.
\end{align}
The emerging semigroups are one-parameter families of nonlinear operators $(T(t))_{t\geq 0}$ acting on a closed subset $G\subseteq X$ with the following properties, see, e.g.\ \cite[Definition 4.4]{Barb10}:
\begin{itemize}
    \item[\rm (i)] $T(t+\tau)x=T(t)(T(\tau)x)$ for $x\in G$, $t,\tau\geq 0$; 
    \item[\rm (ii)] $T(0)x=x$ for $x\in G$;
    \item[\rm (iii)] for any $x\in G$, the function $t\mapsto T(t)x$ is continuous on $[0,\infty)$.
\end{itemize}
The fundamental theorem on the generation of semigroups, see e.g.\ \cite[Proposition 4.2]{Barb10}, states that if $M+\omega I$ is m-accretive for some $\omega\in\R$ then there exists a (nonlinear) semigroup $T_M(t)$ of type $\omega$ on $G=\overline{D(M)}$, \changed{the closure of $D(M)$ in $X$}, which is given by 
\begin{align}
    \label{eq:semigroup}
T_M(t)x=\lim_{n\rightarrow\infty}\left(I+\tfrac{t}{n}M\right)^{-n}x.
\end{align}
\changed{Note that \eqref{eq:M_dissip} implies that $\lambda I+M$ is injective} for all $\lambda>0$ and due to m-accretivity, it is also surjective and therefore bijective. 

 In this work, we will consider \emph{strong solutions} on finite intervals of length $t_f>0$, which are functions $x\in C([0,t_f],X)\cap W^{1,1}((0,t_f],X)$ that are differentiable almost everywhere and satisfy~\eqref{eq:mono_cauchy} pointwise for almost every $t\in(0,t_f)$. For each $x_0\in D(M)$ the strong solutions of \eqref{eq:mono_cauchy} are given by \eqref{eq:semigroup}. 

\medskip 

\noindent \textbf{Inhomogeneous problems.} Here, we regard the solutions of the inhomogeneous Cauchy problem 
\begin{align}
\label{eq:inhomo}
 \tfrac{\mathrm{d}}{\mathrm{d}t}x(t) = -M(x(t))+f(t),\quad x(0)=x_0
\end{align}
for some $f\in L^1([0,t_f],X)$, $t_f>0$. We adapt~\cite[Theorem 4.4]{Barb10} to the special case that $M$ is m-accretive.
\begin{prop}
\label{prop:inhomo}
Let $M$ be an m-accretive operator and let $f\in W^{1,1}([0,t_f],X)$. Then for every $x_0\in D(M)$ there exists a unique strong solution of~\eqref{eq:inhomo} with $x\in W^{1,\infty}([0,t_f],X)$ that fulfills~\eqref{eq:inhomo} for almost every $t\in[0,t_f]$.
\end{prop}


\noindent As a consequence of Proposition \ref{prop:inhomo}, for all $u\in W^{1,1}([0,t_f],U)$ and $x_0\in \dom{M}$ 
there exists a unique strong solution $x\in W^{1,\infty}([0,t_f],X)$ of \eqref{eq:monotone_phs}. Furthermore, this strong solution fulfills for almost every $t\in[0,t_f]$ the \emph{power balance equation}
\begin{align}    
\label{eq:dissip_ineq}
\tfrac12\tfrac{\mathrm{d}}{\mathrm{d}t}\|x(t)\|^2= - \langle x(t),M(x(t))\rangle+\langle u(t),y(t)\rangle_U.
\end{align}

\medskip 

\noindent \textbf{Shifted passivity and stability.} In the following, we are interested in shifted passivity with respect to \emph{steady state pairs} of \eqref{eq:monotone_phs}, i.e.\ $(\overline{x},\overline{u})\in D(M)\times U$ such that $-M(\overline{x})+B\overline{u}=0$. 

\begin{prop}
\label{prop:shifted_passive}
Consider a maximally monotone pH system of the form~\eqref{eq:monotone_phs} with steady state pair $(\overline{x},\overline{u})$ and set $\overline{y} = B^*\overline{x}$. Then, for all $u\in W^{1,1}([0,t_f],U)$ and for almost every $t\in[0,t_f]$, the shifted power balance and shifted passivity equality 
\begin{align*}
\tfrac12\tfrac{\mathrm{d}}{\mathrm{d}t}\|x(t)-\overline{x}\|^2&=-\langle x(t)-\overline{x},M(x(t))-M(\overline{x})\rangle+\langle u(t)-\overline{u},y(t)-\overline{y}\rangle_U\\
&\leq \langle u(t)-\overline{u},y(t)-\overline{y}\rangle_U
\end{align*}
hold.
\end{prop}
\begin{proof}
    As $(\bar x,\bar u)\in \dom{M}\times U$ is a steady state pair, i.e., $\changed{-}M(\bar x) + B\bar u = 0$ and $\tfrac{\mathrm{d}}{\mathrm{d}t} \bar x = 0$, we have
    \begin{align*}
        \tfrac{\mathrm{d}}{\mathrm{d}t} (x(t)-\bar x) = \changed{-}M(x(t)) \changed{+} M(\bar x) + B (u(t)-\bar u).
    \end{align*}
    Thus, the shifted power balance follows from
    \begin{align*}
        \tfrac12 \tfrac{\mathrm{d}}{\mathrm{d}t}\|x(t)-\bar x\|^2 &= \langle x(t)-\bar x,\changed{-}M(x(t)) \changed{+} M(\bar x) + B(u(t) - \bar u)\rangle \\
        &= -\langle x(t)-\bar x, M(x(t)) - M(\bar x)\rangle + \langle B^*(x(t)-\bar x),u(t)-\bar u\rangle_U.
    \end{align*}
    Due to accretivity of $M$, i.e., \eqref{eq:M_dissip}, $-\langle x(t)-\bar x, M(x(t)) - M(\bar x)\rangle \leq 0$, which implies the passivity inequality.
\end{proof}
\noindent \changed{In the following we study uniqueness of steady states and corresponding asymptotic stability if $M$ is \emph{strongly accretive}.
\begin{defn}
An operator $M:X \supseteq D(M)\to X$ is called \emph{strongly accretive at $\overline{x}\in D(M)$}, if there exists $c_M>0$ such that 
\begin{align}
\label{eq:exp_stable}
\langle x-\overline{x},M(x)-M(\overline{x})\rangle \geq c_M\|x-\overline{x}\|^2\quad \text{for all $x\in X$,}
\end{align}
\end{defn}
\noindent For a strongly accretive operator $M$, exponential convergence towards a steady state $\overline{x}$ with a corresponding input $\overline{u}$ can be concluded from Proposition~\ref{prop:shifted_passive} and the Gronwall lemma.}

However, this strong accretivity assumption is only sufficient and not necessary, and in particular too restrictive for our desired application to the primal-dual gradient method. To weaken this assumption, we will consider structured accretive operators $M$ given by
\begin{align}
\begin{split}
\dom{M} &:= \dom{\mathcal{M}_2}\oplus \dom{\mathcal{M}_2^*}\subseteq X_1\oplus X_2,\\
(x_1,x_2)\mapsto M(x_1,x_2)&=\begin{bmatrix}
    \mathcal{M}_1(x_1)-\mathcal{M}_2^*x_2 \\\mathcal{M}_2x_1
\end{bmatrix},\quad (x_1,x_2)\in \dom{M},
\end{split}
\label{def:m_m1_m2}
\end{align}
where $\mathcal{M}_1:X_1\rightarrow X_1$ is bounded and m-accretive 
and $\mathcal{M}_2:X_1\supset D(\mathcal{M}_2)\rightarrow X_2$ is linear, closed and densely defined. 

The following result provides insight into the stability properties of \eqref{eq:monotone_phs} when governed by block operators as introduced above.

\begin{prop}
\label{prop:convergence}
Let $M$ given by \eqref{def:m_m1_m2} satisfy the following conditions:
\begin{itemize}
    \item[(i)] $M$ is m-accretive and invertible;
    \item[(ii)] $\mathcal{M}_1$ is strongly accretive at $\overline{x}_1\in X_1$ and differentiable.
\end{itemize}
Then the following assertion
\begin{itemize}
    \item[(a)] For any $\overline{u}\in U$ there is a unique steady state of \eqref{eq:monotone_phs} defined by $\overline{x}=M^{-1}(B\overline{u})$.
\end{itemize}
For the remaining statements, we consider a constant input $u(t)\equiv \overline{u}\in U$ \changed{and} the steady state $\overline{x}=M^{-1}(B\overline{u})$.
\begin{itemize}
    \item[(b)] The steady state $\overline{x}$ is locally exponentially stable for the dynamics \eqref{eq:monotone_phs} in the following sense: There exists $\delta>0$ and $c>0$ such that for all $x^0 = (x_1^0,x_2^0)\in D(M)$ with $\|x^0_2-\overline{x}_2\|<\delta$ the strong solution fulfills for all $t\geq 0$ 
    \begin{align}
    \label{eq:exp_stable_m1m2}
    \left\| x(t)-\overline{x}
    \right\|
    \leq \alpha e^{-ct}
    \end{align}
    with $\alpha > 0$ depending only on $\|x^0\|$.
    \item[(c)] The steady state $\overline{x}$ 
    is globally asymptotically stable for \eqref{eq:monotone_phs}, i.e., its strong solution satisfies for all initial values $x_0\in D(M)$ $$x(t) = (x_1(t),x_2(t)) \stackrel{t\to \infty}{\to}(\overline{x}_1,\overline{x}_2) = \overline{x}.$$   
    \item[(d)] If $\mathcal{M}_1$ is linear, then the exponential stability from (b) holds globally, i.e., for all $\delta > 0$.
\end{itemize} 
\end{prop}
\begin{proof}
Since $M$ is m-accretive and $B\overline{u}\in W^{1,1}([0,\infty),X)$ \changed{due to $\overline{u}\in U$ being constant in time}, it follows from   Proposition~\ref{prop:inhomo} that the solution $x\in W^{1,\infty}([0,\infty),X)$ of \eqref{eq:monotone_phs} is unique. Furthermore, since $M$ is invertible, it follows that $\overline{x}=M^{-1}(B\overline{u})$ is the unique solution of $0=-M(\overline{x})+B\overline{u}$. This proves assertion~(a). 

To prove the remaining statements, we assume that $u\equiv \overline{u}\in U$ and the shifted states are denoted by \[h=x-\overline{x},\quad \text{and}\quad h_i=x_i-\overline{x}_i,\, i=1,2.\] 
For the strong solutions, we have for almost all $t\geq 0$
\begin{align*}
\frac{1}{2}\frac{\mathrm{d}}{\mathrm{d}t}\left\|\begin{bmatrix}
    h_1\\  h_2
\end{bmatrix}\right\|^2=\left\langle -M(x_1,x_2)+B\overline{u},\begin{bmatrix}
    h_1\\ h_2
\end{bmatrix}\right\rangle
&=\left\langle -M(x_1,x_2)+M(\overline{x}),\begin{bmatrix}
    h_1\\ h_2
\end{bmatrix}\right\rangle\\&=\left\langle \begin{bmatrix}
    \mathcal{M}_1(\overline{x}_1)-\mathcal{M}_1(x_1)+\mathcal{M}_2^*h_2\\ -\mathcal{M}_2h_1
\end{bmatrix},\begin{bmatrix}
    h_1\\ h_2
\end{bmatrix}\right\rangle\\&=\langle \mathcal{M}_1(\overline{x}_1)-\mathcal{M}_1(x_1), h_1\rangle_{X_1}\\&\leq -c_1\| h_1\|_{X_1}^2\leq 0,
\end{align*}
\changed{where the final step uses (ii).}
Hence the norm is non-increasing. Furthermore, integration results in 
\begin{align*}
\|h_1(t)\|_{X_1}^2-\left\|\begin{bmatrix}
    h_1(0)\\ h_2(0)
\end{bmatrix}\right\|^2\leq \left\|\begin{bmatrix}
    h_1(t)\\ h_2(t)
\end{bmatrix}\right\|^2-\left\|\begin{bmatrix}
    h_1(0)\\ h_2(0)
\end{bmatrix}\right\|^2\leq - 2\int_0^tc_1\|h_1(\tau)\|_{X_1}^2\mathrm{d}\tau,
\end{align*}
which holds for all $t\geq 0$. Applying the Gronwall Lemma from \cite[Lemma 4.1]{GrueJ15} leads to the estimate
\begin{align}
\label{eq:x1_exp_stable}
    \|h_1(t)\|_{X_1}\leq \left\|\begin{bmatrix}
h_1(0)\\ h_2(0)        
    \end{bmatrix}\right\|e^{-c_1t},\quad \forall\,t\geq 0.
\end{align}
This implies that $\|h_1(t)\|_{X_1}\rightarrow0$ as $t \rightarrow\infty$ and as $t\mapsto\left\|\sbvek{h_1(t)}{h_2(t)}\right\|$ is non-increasing, it follows that the function $t\mapsto\|h_2(t)\|_{X_2}$ is bounded on $[0,\infty)$.

We proceed by showing the local exponential stability of the origin claimed in (b). Using the derivative, we find 
\begin{align}
\label{eq:lin_approx}
M(\overline{x}+h)=M(\overline{x})+\mathrm{D}M(\overline{x})h+r(h_1), \quad \mathrm{D}M(\overline{x})h=\begin{bmatrix}
    \mathrm{D}\mathcal{M}_1(\overline{x}_1) &-\mathcal{M}_2^*\\ \mathcal{M}_2&0
\end{bmatrix}\begin{bmatrix}
    h_1\\ h_2
\end{bmatrix},
\end{align}
where $r$ satisfies $\frac{\|r(h_1)\|}{\|h_1\|_{X_1}}\rightarrow 0$ as $\|h_1\|_{X_1}\rightarrow 0$. This implies \changed{using skew-symmetry of the off-diagonal part of $\mathcal{M}$}, that
\begin{align}
\langle h_1,\mathrm{D}\mathcal{M}_1(\overline{x}_1)h_1+r(h_1)\rangle_{X_1}&=\langle h,\mathrm{D}M(\overline{x})h+r(h_1)\rangle\nonumber\\
&=\langle h, M(\overline{x}+h)-M(\overline{x})\rangle\nonumber\\
&=\langle h_1,\mathcal{M}_1(\overline{x}_1+h_1)-\mathcal{M}_1(\overline{x}_1)\rangle_{X_1}\geq c_1\|h_1\|_{X_1}^2.\label{eq:laststep}
\end{align}
Let $h_1$ be small enough such that $\tfrac{\|r(h_1)\|}{\|h_1\|_{X_1}}\leq\tfrac{c_1}{2}$ and set $z=\tfrac{h_1}{\|h_1\|_{X_1}}\in X_1$, then $\changed{\|z\|_{X_1}}=1$. Therefore we have for all $z\in X_1$ with $\|z\|_{X_1}=1$
\[
\langle z,\mathrm{D}\mathcal{M}_1(\overline{x}_1)z\rangle_{X_1}=\langle \tfrac{h_1}{\|h_1\|_{X_1}},\mathrm{D}\mathcal{M}_1(\overline{x}_1)\tfrac{h_1}{\|h_1\|_{X_1}}\rangle_{X_1}\geq c_1-\langle \tfrac{h_1}{\|h_1\|_{X_1}},\tfrac{r(h_1)}{\|h_1\|_{X_1}}\rangle_{X_1}\geq \tfrac{c_1}{2},
\]
\changed{where we used \eqref{eq:laststep}.}
Hence, $\mathrm{D}\mathcal{M}_1$ is strongly accretive as for any $z\in X_1\setminus\{0\}$,
\begin{align*}
    \langle z,\mathrm{D}\mathcal{M}_1(\overline{x}_1)z\rangle_{X_1} = \|z\|_{X_1}^2\langle \tfrac{z}{\|z\|_{X_1}}, \mathrm{D}\mathcal{M}_1(\overline{x}_1)\tfrac{z}{\|z\|_{X_1}}\rangle_{X_1} \geq \tfrac{c_1}{2}\|z\|_{X_1}^2.
\end{align*}
Furthermore, as $M$ is invertible, $\mathcal{M}_2$ is necessarily surjective and therefore, it follows from \cite[Proposition 2.9]{GernHins22} that $\mathrm{D}M(\overline{x})$ generates an exponentially stable semigroup. Then \cite[Corollary 2.2]{Kato95} implies that the origin is a locally exponentially stable equilibrium for \eqref{eq:mono_cauchy}, that is, (b). Further, the claim (d) is also immediately clear as in case of linearity of $\mathcal{M}_1$ (and hence of $M$), the remainder term satisfies $r\equiv 0$.

To see (c), we will show that the trajectory $t\mapsto(h_1(t),h_2(t))$ reaches any neighborhood of the origin, i.e., in particular the neighborhood on which we have the local exponential stability from (b). Since we already have $x_1(t)\rightarrow \overline{x}_1$, we remain to show that for all $\varepsilon>0$ there exists $t\geq 0$ such that $\|h_2(t)\|_{X_2}<\varepsilon$ holds. Towards showing a contradiction, we assume that there exists $\varepsilon>0$ such that for all $t\geq 0$ one has $\|h_2(t)\|_{X_2}\geq\varepsilon$.  We consider the Lyapunov function of the linearized system given by a positive solution $P\in L(X)$ of the Lyapunov equation
\[
\langle \mathrm{D}M(\overline{x})h,Ph\rangle+\langle Ph,\mathrm{D}M(\overline{x})h\rangle=-\langle h,h\rangle,
\]
which exists due to \cite[Theorem 5.1.3]{CurtZwar95}. 
Hence, using \eqref{eq:x1_exp_stable} and \eqref{eq:lin_approx}, we can choose $t>0$ sufficiently large such that 
$\|r(h_1(t))\|\leq (4\|P\|)^{-1}\|h_1(t)\|_{X_1}$,
and we estimate for sufficiently large $t$
\begin{align*}
\tfrac{\mathrm{d}}{\mathrm{d}t}\langle Ph(t),h(t)\rangle &= \langle Ph(t),\dot x(t)\rangle+\langle \dot x(t),Ph(t)\rangle\\&=\langle Ph(t),\mathrm{D}M(\overline{x})h(t)+r(h_1(t))\rangle+\langle \mathrm{D}M(\overline{x})h(t)+r(h_1(t)),Ph(t)\rangle
\\&\leq \langle Ph(t),\mathrm{D}M(\overline{x})h(t)\rangle+\langle \mathrm{D}M(\overline{x})h(t),Ph(t)\rangle + 2\|P\|\|r(h_1(t))\|\|h_1(t)\|_{X_1}
\\&\leq -\frac{1}{2}\|h(t)\|^2.
\end{align*}
Integrating the above estimte and using that $t\mapsto\|h(t)\|$ is non-increasing on $[0,\infty)$ as shown in the first part, this implies
\[
\int_0^\infty\|h(\tau)\|^2{\rm d}\tau<\infty.
\]
This in combination with $\|h_2(t)\|_{X_2}\geq \varepsilon$ for all $t\geq 0$ leads to the desired contradiction.
\end{proof}

Next, we prove the closedness of the presented class of nonlinear monotone pH systems. 
\begin{prop}
\label{prop:coupling}
Consider two monotone pH systems of the form \eqref{eq:monotone_phs} given by $(M_1,B_1)$ and $(M_2,B_2)$ in the Hilbert spaces $X_1$ and $X_2$, \changed{with inputs and outputs  $u_i,y_i\in U$, $i=1,2$.} Further, let a decomposition of the inputs $u_i=\begin{smallbmatrix}u_i^1\\ u_i^2\end{smallbmatrix}\in U=U_1\times U_2$ and outputs $y_i=\begin{smallbmatrix}y_i^1 \\ y_i^2\end{smallbmatrix}\in U_1\times U_2$ and  $B_i=\begin{smallbmatrix} B_i^1 & B_i^2\end{smallbmatrix}$, $i=1,2$ be given. Then, for \changed{bounded and linear} $\mathcal{F}:U_2\rightarrow U_1$, the interconnection 
\begin{align*}
\begin{bmatrix}
        0 &\mathcal{F}\\
        -\mathcal{F}^* & 0
    \end{bmatrix}    \begin{bmatrix}
        y_1^1\\
        y_2^1
    \end{bmatrix} = 
    \begin{bmatrix}
        u_1^1\\
        u_2^1
    \end{bmatrix}
\end{align*}
leads again to a monotone pH system given by
\begin{align}
\label{eq:coupled_ph}
\begin{split}
\begin{bmatrix}\dot x_1\\ \dot x_2\end{bmatrix}&=\underbrace{\begin{bmatrix}
    M_1(x_1)\\M_2(x_2)
\end{bmatrix}+\begin{bmatrix} 0& B_1^1\mathcal{F}(B_2^1)^*\\-B_2^1\mathcal{F}^*(B_1^1)^*&0\end{bmatrix}\begin{bmatrix} x_1\\ x_2\end{bmatrix}}_{=:M(x_1,x_2)} + \begin{bmatrix}B_1^2&0\\0&B_2^2\end{bmatrix}\begin{bmatrix}u_1^2\\ u_2^2\end{bmatrix}\\ \begin{bmatrix}y_1^2\\y_2^2\end{bmatrix}&=\begin{bmatrix}(B_1^2)^*&0\\0&(B_2^2)^*\end{bmatrix}\begin{bmatrix} x_1\\ x_2\end{bmatrix}.
\end{split}
\end{align}
Furthermore, if the pH systems given by $(M_1,B_1)$ and $(M_2,B_2)$ are maximally monotone, then the interconnected system is a~maximally monotone pH system. 
Moreover, if $M_1$ and $M_2$ are strongly accretive at $\overline{x}_1\in X_1$ and $\overline{x}_2\in X_2$, respectively, then $M$ corresponding to the interconnected system \eqref{eq:coupled_ph} is strongly accretive at $\bar{x} = (\overline{x}_1,\overline{x}_2)\in X_1\times X_2$. Therefore if $(u_1^2,u_2^2)=0$, then $(x_1(t),x_2(t))\rightarrow (\overline{x}_1,\overline{x}_2)$ exponentially as $t\rightarrow\infty$ for any strong solution of \eqref{eq:coupled_ph}.
\end{prop}
\begin{proof}
To see that the interconnected pH system is monotone, one easily checks \eqref{eq:M_dissip}. The maximality can be treated using a perturbation argument from \cite[Theorem 3.1]{Barb10}: If $\mathcal{A}$ is m-accretive in $X$ and $\mathcal{B}:X\rightarrow X$ is bounded and m-accretive, then $\mathcal{A}+\mathcal{B}$ is m-accretive. We want to apply this result to $$\mathcal{A}=\begin{bmatrix}
M_1\\ M_2   
\end{bmatrix} \quad \mathrm{and}\quad \mathcal{B}=\begin{bmatrix}
    0& B_1^1\mathcal{F}(B_2^1)^*\\ -B_2^1\mathcal{F}^*(B_1^1)^*&0
\end{bmatrix}.$$ 
To this end, we note that $\mathcal{B}$ is skew-adjoint and therefore accretive and $(-\infty,0]\subseteq\rho(\mathcal{B})$ holds. Hence $\lambda I+\mathcal{B}$ is surjective for all $\lambda >0$ which means that $\mathcal{B}$ is m-accretive.

If $M_1$ and $M_2$ \changed{are} strongly accretive, then due to the skew-symmetric interconnection, we have
\[
\langle M(x_1,x_2),(x_1,x_2)\rangle=\langle M_1(x_1),x_1\rangle+\langle M_2(x_2),x_2\rangle\geq \min\{c_{M_1},c_{M_2}\}\left\|\begin{bmatrix}
    x_1\\ x_2
\end{bmatrix}\right\|^2,
\]
and therefore the operator $M$ is strongly accretive. Thus, convergence follows from Proposition~\ref{prop:shifted_passive} and the Gronwall lemma.
\end{proof}

\section{Infinite-dimensional optimizer dynamics as a monotone pH system}
\label{sec:optimizer_dynamics}
\noindent In this part, we will first show that the continuous-time primal-dual method for continuous-time optimal control problems falls into the class of monotone pH systems presented in the previous section. Then, we endow the system with suitable port variables which serves as a foundation for optimization-based control by interconnection. 

We consider optimal control problems on finite time horizons $t_f>0$ subject to linear ordinary differential equations (ODE) of the form
\begin{align}\label{eq:oc_with_dynamics}
\begin{split}
        \min_{(x,u)\in H^1([0,t_f],\R^n)\times L^2([0,t_f],\R^m)} J(x,u) \quad \mathrm{s.t.}\quad \dot x(\tau) = Ax(\tau) + Bu(\tau)+f(\tau),\quad x(0)=x^0.
\end{split}
\end{align}
Here, $A\in\R^{n\times n}$, $B\in \R^{n\times m}$, $f\in L^2([0,t_f],\R^n)$, $x^0\in \R^n$ and the cost functional is defined by 
\begin{align}
    \label{def:J}
J: L^2([0,t_f],\R^n)\times L^2([0,t_f],\R^m)\rightarrow\R,\   J(x,u):= \int_0^{t_f} \ell_x(x(\tau))+\frac{\alpha}{2}\|u(\tau)\|^2\,\mathrm{d}\tau
\end{align}
where $\alpha>0$ and $\ell_x:\R^n\rightarrow[0,\infty)$ is continuous, convex and with $\ell_x(x)\in L^1([0,t_f],\R)$ for all $x\in L^2([0,t_f],\R^n)$ and such that the $L^2$-gradient $\nabla_x J(x)$ of the mapping $x\mapsto \int_0^{t_f}\ell_x(x(\tau))\,\mathrm{d}\tau$ exists for all $x\in L^2([0,t_f],\R^n)$. Clearly, the mapping $u\mapsto\tfrac{\alpha}{2}\|u\|_{L^2}^2$ is differentiable with gradient $\nabla_u J(u)=\alpha u$. Therefore, $J$ is differentiable with gradient \[\nabla J(x,u)=\begin{bmatrix}\nabla_x J(x)\\\alpha u\end{bmatrix}.\]  
For brevity of notation, we will identify $L^2([0,t_f],\R^n)\times L^2([0,t_f],\R^m)$ with
$L^2([0,t_f],\R^n\times \R^m)$ in the following.

To show existence of solutions of \eqref{eq:oc_abstract}, we pursue the common strategy of eliminating the state. 
This is achieved by applying the variation of constants formula, leading to the input-to-state map
\begin{align}
\label{eq:input_2_state}
\begin{split}
&    I_{x_0,f}:L^2([0,t_f],\R^m)\rightarrow H^1([0,t_f],\R^n),\\  &I_{x_0,f}(u)(t):=e^{At}x_0+\int_0^te^{A(t-\tau)}(Bu(\tau)+f(\tau)){\rm d}\tau=x(t).
\end{split}
\end{align}
With this we may define the cost function only depending on the input:
\begin{align}
    \label{def:J_u}
J_u:L^2([0,t_f],\R^m)\rightarrow[0,\infty),\quad u\mapsto J_u(u):=\int_0^{t_f}\ell_x(I_{x_0,f}(u)(\tau))+\frac{\alpha}{2}\|u(\tau)\|^2{\rm d}\tau.
\end{align}
We will use the cost function $J_u$ to prove existence of a minimizer for \eqref{eq:oc_abstract}.



 \begin{lem} 
 \label{lem:cost}
 Let $J$ and $J_u$ be given by \eqref{def:J} and \eqref{def:J_u}, respectively. Then:
 \begin{itemize}
     \item[\rm (a)] $J$ is convex;
     \item[\rm (b)] $J_u$ is strictly convex, continuous  and coercive.
 \end{itemize}
 \end{lem}
 \begin{proof}
The convexity of $J$ follows from linearity and monotonicity of integration. \changed{Note that $I_{x_0,f}(\beta u_1+(1-\beta)u_2)(t)=\beta I_{x_0,f}(u_1)(t)+(1-\beta)I_{x_0,f}(u_2)(t)$} holds for all $\beta\in[0,1]$ and all $u_1,u_2\in L^2([0,t_f],\R^m)$. Furthermore, $J_u$ is continuous as the concatenation of the continuous functions $F_1:L^2([0,t_f],\R^m)\rightarrow C([0,t_f],\R^n)$, \changed{$u\mapsto I_{x_0,f}(u)$} and $F_2:C([0,t_f],\R^n)\rightarrow\R$, $x\mapsto\int_0^{t_f}\ell_x(x(\tau))\mathrm{d}\tau$. 
The strict convexity of $J_u$ follows from the convexity of the first term $\int_0^{t_f}\ell_x(I_{x_0,f}(u)(\tau))\,\mathrm{d}\tau$ \changed{(as its integrand is the concatenation of a convex and a linear function, hence convex, and due to monotonicity of the integral)} and the fact that $u\mapsto\int_0^{t_f}\tfrac{\alpha}{2}\|u(\tau)\|^2\mathrm{d}\tau$ is strictly convex. Finally, the coercivity follows as $J_u$ is the sum of a nonnegative function (due to $\ell_x(x)\geq 0$) and the coercive function $\int_0^{t_f}\tfrac{\alpha}{2}\|u(\tau)\|^2\,\mathrm{d}\tau$.
 \end{proof}
 
\noindent Having discussed existence of solutions, we now derive optimality conditions. To this end, we treat the two constraints of the problem \eqref{eq:oc_with_dynamics}, encoding the dynamics of the ODE and the initial condition, by means of the unbounded operator
   \begin{align}
   \label{eq:constraint}
   \begin{split}
    \calC&: L^2([0,t_f],\R^{n}\times \R^m) \supset D(\mathcal{C}) \to L^2([0,t_f],\R^n) \times \R^n,\\
    D(\calC)&:= H^1([0,t_f],\R^n)\times L^2([0,t_f],\R^m),\quad 
    \calC\changed{\bvek{x}{u}} = \begin{bmatrix}\tfrac{\mathrm{d}}{\mathrm{d}\tau} x - Ax - Bu\\x(0)\end{bmatrix}
       \end{split}
\end{align} 
where $Ax+Bu$ has to be viewed as a bounded linear operator on $L^2([0,t_f],\R^n \times \R^m)$. Below, we collect some properties of $\mathcal{C}$ and its adjoint $\calC^*$.
 \begin{lem}
 \label{lem:C}
 The operator $\calC$ as defined in \eqref{eq:constraint} is densely defined and closed. Its adjoint $\calC^*$ is given by 
     \begin{align*}
        \calC^*: L^2([0,t_f],\R^n)\times \R^n \supset D(\mathcal{C}^*) \to L^2([0,t_f],\R^n \times\R^m),\quad \calC^*\changed{\bvek{\lambda}{\lambda_0}}= \begin{bmatrix}
        -\frac{\mathrm{d}}{\mathrm{d}\tau} \lambda - A^\top \lambda\\
        -B^\top \lambda
    \end{bmatrix}\\
    D(\calC^*)= \{(\lambda,\lambda_0)\in L^2([0,t_f],\R^n)\times \R^n\,|\,\lambda \in H^1([0,t_f],\R^n)\wedge \lambda(0)=\lambda_0\wedge \lambda(T)=0\}. 
    \end{align*}
    Furthermore, $\ran\calC=L^2([0,t_f],\R^n)\times \R^n$, i.e.\ $\calC$ is onto and  $\ker\calC^*=\{0\}$.
 \end{lem}
 \begin{proof}
The proof is presented in Appendix~\ref{sec:proof_lemma_c}.
 \end{proof}
\noindent Using the operator~$\calC$ given by \eqref{eq:constraint}, we can concisely write the optimal control problem \eqref{eq:oc_with_dynamics} as 
\begin{align}\label{eq:oc_abstract}
    \min_{(x,u)\in H^1([0,t_f],\R^n)\times L^2([0,t_f],\R^m)} J(x,u), \quad \changed{\mathrm{s.t.}}\quad \calC\begin{bmatrix}x\\ u \end{bmatrix}=\begin{bmatrix}f\\ x^0\end{bmatrix}.
\end{align}



\noindent We briefly recall a standard result from optimization \cite{EkelTema99}.
\begin{lem}\label{lem:eketem}
    Let $X$ be a Hilbert space, $F:X \to \R$ be Fréchet differentiable and $C\subset X$ be convex. Then, any solution $x^* \in \arg\min_{x\in C}F(x)$ satisfies $\langle F'(x^*), x-x^*\rangle \geq 0$ for all $x\in C$. If $C = \{x^*\} + V$ for some subspace $V$ of $X$, then $\langle F'(x^*),v\rangle = 0$ for all $v \in V$. 
\end{lem}
\begin{proof}
    The first claim follows from \cite[Proposition 2.1]{EkelTema99}. For the second claim, we observe that $\langle F'(x^*),x-x^*\rangle \geq 0$ for all $x\in C = \{x^*\}+V$ if and only if $\langle F'(x^*),v\rangle \geq 0$ for all $v\in V$. As $V$ is a subspace, $-v\in V$ for all $v\in V$ such that $-\langle F'(x^*),v\rangle \geq 0$ for all $v\in V$. This implies that $\langle F'(x^*),v\rangle = 0$ for all $v\in V$.
\end{proof}

\noindent Using Lemma~\ref{lem:eketem}, we obtain the following optimality conditions. 
\begin{prop}
\label{prop:optimality_condition}
    There exists a unique optimal solution $(\hat x,\hat u) \in \dom \calC$ for \eqref{eq:oc_abstract}. Further, there exists a Lagrange multiplier $(\hat \lambda, \hat \lambda_0)\in \dom{\calC^*}$ such that
    \begin{align}
    \label{eq:optimality_conditions}
M_\mathrm{opt}(\hat x,\hat u,(\hat \lambda,\hat \lambda_0)) :=  \begin{bmatrix}
       \begin{bmatrix}\nabla_xJ(\hat x,\hat u) \\ \nabla_uJ(\hat x,\hat u)\end{bmatrix} - \calC^*\sbvek{\hat\lambda}{\hat\lambda_0} \\
        \calC \sbvek{\hat x}{\hat u}
    \end{bmatrix}
    = 
        \begin{bmatrix}
            0\\0\\\sbvek{f}{x^0}
        \end{bmatrix}.
\end{align}
Moreover, the operator $M_\mathrm{opt}:\dom\calC\times\dom{\calC^*}\rightarrow$ 
$L^2([0,t_f],\R^n\times\R^m\times\R^n)\times\R^n$ 
is bijective and m-accretive. In particular, the solution $(\hat x,\hat u,(\hat \lambda,\hat \lambda_0))$ of \eqref{eq:optimality_conditions} is unique.
\end{prop}
\begin{proof}
The existence of the optimal solution $(\hat x,\hat u) \in \dom \calC$ for \eqref{eq:oc_abstract} follows from the fact that $J_u$ is continuous and coercive by Lemma~\ref{lem:cost}, cf.~\cite[Proposition 1.2]{EkelTema99}. Uniqueness of the optimal solution $(\hat x,\hat u)$ follows by strict convexity of $J_u$. To derive the optimality conditions \eqref{eq:optimality_conditions} we use a standard argumentation, see e.g.~\cite{schiela2013concise}, which we briefly present in our context for completeness. First, as $\ker \mathcal{C}$ is a linear space, by Lemma~\ref{lem:eketem} this implies
    \begin{align*}
        \langle \nabla J(\hat x,\hat u), v\rangle = 0 \qquad \forall v\in \ker \calC\subseteq L^2([0,t_f],\R^{n}\times \R^m).
    \end{align*}
    and hence
\begin{align*}
    \nabla J(\hat x,\hat u) \in (\ker \calC)^\perp= \{\ell \in L^2([0,t_f],\R^n \times \R^m)~\,|\,~ \langle \ell, v\rangle = 0\ \forall v\in \ker \calC\}.
\end{align*}
As $\calC$ has closed range due to its surjectivity (see Lemma~\ref{lem:C}), the closed range theorem implies $(\ker \calC)^\perp = \ran \calC^*$. That is, there exists $(\hat \lambda,\hat\lambda_0)\in \dom {\calC^*}$ such that
\begin{align*}
    \begin{bmatrix}
        \nabla_xJ(\hat x,\hat u)\\
        \nabla_uJ(\hat y,\hat u)
    \end{bmatrix} = \calC^*\changed{\bvek{\hat\lambda}{\hat\lambda_0}}.
\end{align*}
This proves existence of $(\hat \lambda,\hat\lambda_0)\in\dom {\calC^*}$ such that the optimality conditions \eqref{eq:optimality_conditions} hold.

We continue by showing injectivity of $M_\mathrm{opt}$. Let $(\hat x_1,\hat u_1,\hat\lambda_1,\hat\lambda_{0,1})$ and $(\hat x_2,\hat u_2,\hat\lambda_2,\hat\lambda_{0,2})$ be two solutions of \eqref{eq:optimality_conditions}. Then we obtain using the convexity of $J$,
\begin{align*}
0&=\langle M_\mathrm{opt}(\hat x_1,\hat u_1,\hat\lambda_1,\hat\lambda_{0,1})-M_\mathrm{opt}(\hat x_2,\hat u_2,\hat\lambda_2,\hat\lambda_{0,2}),(\hat x_1-\hat x_2,\hat u_1-\hat u_2,\hat\lambda_1,\hat\lambda_{0,1})\rangle\\
&=\left\langle \begin{bmatrix}
    \nabla_xJ(\hat x_1)-\nabla_xJ(\hat x_2)\\ 
    \nabla_uJ(\hat u_1)-\nabla_uJ(\hat u_2)
\end{bmatrix},\begin{bmatrix}
\hat x_1-\hat x_2 \\ \hat u_1-\hat u_2  \end{bmatrix}\right\rangle\\
&=\left\langle 
    \nabla_xJ(\hat x_1)-\nabla_xJ(\hat x_2),
\hat x_1-\hat x_2
\right\rangle+\langle 
    \nabla_uJ(\hat u_1)-\nabla_uJ(\hat u_2),\hat u_1-\hat u_2\rangle\\
    &\geq \alpha\|\hat u_1-\hat u_2\|^2,
\end{align*}
\changed{where in the last inequality we used the strict convexity and \eqref{def:convex_deriv}.}
Therefore $\hat u_1=\hat u_2$ and using the input-to-state map $I_{x_0,f}$ defined in~\eqref{eq:input_2_state}, we conclude that $\hat x_1=I_{x_0,f}(\hat u_1)=I_{x_0,f}(\hat u_2)=\hat x_2$.  
Furthermore, Lemma~\ref{lem:C} gives $\ker\calC^*=\{0\}$ and therefore, also $(\hat \lambda,\hat \lambda_0)$ is uniquely determined. Therefore, $M_\mathrm{opt}$ is injective.

Next, we prove surjectivity of $M_\mathrm{opt}$. Let $\changed{(l_x,l_u,f,x^0)\in L^2([0,t_f],\R^n\times\R^m\times\R^n)\times\R^n}$ be given and define the modified costs
    \begin{align*}
        \tilde J (x,u) := J(x,u) - \langle l_x,x\rangle
        - \langle l_u,u\rangle.
    \end{align*}
    The modified reduced cost $\tilde J_u (u):=\tilde J (I_{x_0,f}(u),u)$ clearly is continuous. Further, due to the quadratic control term in $J(x,u)$, the linearly perturbed cost $\tilde J_u(u)$ is still coercive. Thus, there is an optimal solution $(\tilde x,\tilde u)$ and associated Lagrange multiplier $(\tilde \lambda,\tilde \lambda_0)\in\dom{\calC^*}$ which fulfills the optimality conditions  \eqref{eq:optimality_conditions} with $J$ replaced by $\tilde J$, \changed{that is,
    	\begin{align*}
    	\begin{bmatrix}
    		\begin{bmatrix}\nabla_x\tilde J(\tilde x,\tilde u) \\ \nabla_u\tilde J(\tilde x,\tilde u)\end{bmatrix} - \calC^*\sbvek{\tilde\lambda}{\tilde\lambda_0} \\
    		\calC \sbvek{\tilde x}{\tilde u}
    	\end{bmatrix}
    	= 
    	\begin{bmatrix}
    		0\\0\\\sbvek{f}{x^0}
    	\end{bmatrix}.
  		\end{align*}
} Since $\nabla \tilde J_x = \nabla J_x -l_x$ and $\nabla \tilde J_u = \nabla J_u-l_u$ holds, this is equivalent to 
\changed{
    	\begin{align*}
\begin{bmatrix}
\begin{bmatrix}\nabla_x J(\tilde x,\tilde u) \\ \nabla_u J(\tilde x,\tilde u)\end{bmatrix} - \calC^*\sbvek{\tilde\lambda}{\tilde\lambda_0} \\
\calC \sbvek{\tilde x}{\tilde u}
\end{bmatrix}
= 
\begin{bmatrix}
l_x\\l_u\\\sbvek{f}{x^0}
\end{bmatrix}.
\end{align*}}
 which, \changed{as $(l_x,l_u,f,x^0)$ was arbitrary,} proves that $M_\mathrm{opt}$ is onto and thus bijective. 
    
    Finally, we show that $M_\mathrm{opt}$ is m-accretive. By convexity of the cost, $M_\mathrm{opt}$ fulfills \eqref{eq:M_dissip} and is therefore is accretive. Since $M_\mathrm{opt}$ is bijective, we have $0\in\rho(M_\mathrm{opt})$ and since the resolvent set is open~\cite{Kato95}, we have that $\lambda I-M_\mathrm{opt}$ is invertible for $\lambda>0$ sufficiently small. Hence, $\lambda I-M_\mathrm{opt}$ 
 is surjective which means that $M_\mathrm{opt}$ is m-accretive.
\end{proof}

\noindent The operator $M_\mathrm{opt}$ that is given by \eqref{eq:optimality_conditions} gives rise to the following infinite-dimensional primal-dual gradient dynamics
\begin{align}\label{eq:graddyn}
  \frac{\mathrm{d}}{\mathrm{d}t} \begin{bmatrix}
      \sbvek{x(t)}{u(t)}\\
      \sbvek{\lambda(t)}{\lambda_0(t)}
  \end{bmatrix}
  = \begin{bmatrix}
        0&\calC^*\\
        -\calC&0
    \end{bmatrix}
    \begin{bmatrix}
              \sbvek{x(t)}{u(t)}\\
\sbvek{\lambda(t)}{\lambda_0(t)}
    \end{bmatrix}
    + \begin{bmatrix}
        -\nabla J(x(t),u(t))\\
        \sbvek{f}{x_0}
    \end{bmatrix},
\end{align}
which we will call \emph{optimizer dynamics}. These may be seen as a gradient descent in the primal variable $(x,u)$ and a gradient ascent in the dual variable $(\lambda,\lambda_0)$. As we shall see later, the unique steady state of \eqref{eq:graddyn} is a solution of the optimality system~\eqref{eq:optimality_conditions}. By strict convexity of the cost function, this is the unique solution to the optimal control problem.

To formulate the above dynamics as monotone pH system as defined in Definition~\ref{def:mono_phs}, we rewrite \eqref{eq:graddyn} as
\begin{subequations}\label{eq:graddyn_LQ}
    \begin{align}\label{eq:graddyn_LQ_state}
  \frac{\mathrm{d}}{\mathrm{d}t} \begin{bmatrix}
      \sbvek{x(t)}{u(t)}\\
      \sbvek{\lambda(t)}{\lambda_0(t)}
  \end{bmatrix}
  &= \begin{bmatrix}
      -\nabla J(x(t),u(t))\\0
  \end{bmatrix}+\begin{bmatrix}
     0&\calC^*\\
        -\calC&0
    \end{bmatrix}
    \begin{bmatrix}
        \sbvek{x(t)}{u(t)}\\\sbvek{\lambda(t)}{\lambda_0(t)}
    \end{bmatrix}
    + \underbrace{\begin{bmatrix}
        0\\
        I_{2n}
    \end{bmatrix}}_{=:\mathcal{B}}
        u_\mathrm{opt}(t),
\intertext{ with $u_\mathrm{opt}(t)  
\equiv \sbvek{f}{x^0}$ and the collocated output}  
    y_{\mathrm{opt}}(t)
    &= \mathcal{B}^*\begin{bmatrix}
        \sbvek{x(t)}{u(t)}\\
        \sbvek{\lambda(t)}{\lambda_0(t)}
    \end{bmatrix} = \begin{bmatrix}
        0&I_{2n}
    \end{bmatrix}\begin{bmatrix}
        \sbvek{x(t)}{u(t)}\\
        \sbvek{\lambda(t)}{\lambda_0(t)}
    \end{bmatrix}
    =  \begin{bmatrix}
        \lambda(t)\\
        \lambda_0(t)
    \end{bmatrix}.\label{eq:graddyn_LQ_output}
\end{align}
\end{subequations}
\changed{Note that we assume only $f\in L^2([0,t_f],\mathbb{R}^m)$ in the constraints of the optimization problem~\eqref{eq:oc_with_dynamics}, but when considering the gradient optimizer dynamics~\eqref{eq:graddyn}, we have in fact an inhomogeneity that is constant in the virtual time. Therefore, the regularity of the inhomogeneity $\mathcal{B} u_\mathrm{opt}\in W^{1,1}([0,t_f],X)$ is fulfilled for the existence of strong solutions of \eqref{eq:graddyn_LQ} in the sense of Proposition~\ref{prop:inhomo}.}

Since $M_\mathrm{opt}$ is m-accretive, we immediately obtain from Proposition~\ref{prop:convergence} and Proposition~\ref{prop:optimality_condition}  the following convergence result.
\begin{cor}
\label{lem:graddyn_ph}
    The system \eqref{eq:graddyn_LQ} is a maximally monotone pH system in the sense of Definition~\ref{def:mono_phs}. Furthermore, assume that $J$ is twice differentiable and locally strictly convex with quadratic growth, i.e., there exists some $c_{\hat x}>0$, such that for all $x$ in a neighborhood of $\hat x$, the following holds
    \begin{align}
    \label{eq:convex_in_x}
        \langle \nabla_xJ(\hat x)-\nabla_xJ( x),x-\hat x\rangle\geq c_{\hat x}\|x-\hat x\|^2.
    \end{align}
    Then, for all initial values $(x(0),u(0),\lambda(0),\lambda_0(0))\in \dom{\mathcal{C}}\times \dom{\mathcal{C}^*}$, the strong solutions of the optimizer dynamics \eqref{eq:graddyn} satisfy
    \begin{align*}
        (x(t),u(t),\lambda(t),\lambda_0(t))\stackrel{t\to \infty}{\to} (\hat x, \hat u,\hat \lambda,\hat\lambda_0),
    \end{align*}
    where $(\hat x, \hat u,\hat \lambda,\hat\lambda_0)$ is the solution of the optimality conditions~\eqref{eq:optimality_conditions}, that is, $(\hat x,\hat u)$ is the unique optimal solution of~\eqref{eq:oc_abstract}. Further, there exists $\delta > 0$ such that if the initial Lagrange multiplier satisfies $\|(\lambda(0)-\hat\lambda,\lambda_0(0)-\hat\lambda_0)\|_{L^2([0,t_f],\R^n)\times \R^n}<\delta$, the convergence is exponential. 
\end{cor}
Corollary~\ref{lem:graddyn_ph} extends the known convergence result in finite-dimensions \cite{StegPers15} to the infinite-dimensional case. Whereas the proof in \cite{StegPers15} is based on an application of a generalized LaSalle's invariance principle~\cite{CherMall2016}, we use the particular monotone pH structure to avoid compactness assumptions usually present in LaSalle's invariance principle.

\section{Control-by-interconnection via optimizer dynamics}
\label{sec:coupled_optimizer_plant}
\noindent In the following, we provide a control-by-interconnection approach by coupling the monotone pH optimizer dynamics \eqref{eq:graddyn} with the monotone pH plant dynamics of the form~\eqref{eq:monotone_phs}. In this context, we will perform a coupling in a power-preserving way leveraging Proposition~\ref{prop:coupling}.




\noindent The plant to be controlled is assumed to be monotone pH system on $\R^n$ given by 
\begin{align}\label{eq:plant}
    \dot x_p(t)=-M_p(x_p(t))+B_pu(t),\quad  y_p(t)=B_p^\top x(t),
\end{align}
where $B_p\in \R^{n\times m}$ and $M_p:\R^n \to \R^n$ is an accretive operator. To couple the above plant~\eqref{eq:plant} (with $m$-dimensional ports) with the optimizer dynamics~\eqref{eq:graddyn_LQ} (with $2n$-dimensional ports), we utilize the skew-symmetric interconnection
\begin{align}\label{eq:coupling}
\begin{bmatrix}
    u_{\mathrm{opt},1}\\
    u_{\mathrm{opt},2}\\
        u_p
\end{bmatrix}
 = \frac{1}{\alpha}\begin{bmatrix}
    0&0&0\\
    0& 0 & B \\
     0&-B^\top & 0
 \end{bmatrix}
 \begin{bmatrix}
    y_{\mathrm{opt},1}\\
     y_{\mathrm{opt},2}\\
     y_p
 \end{bmatrix}
\end{align}
with the partitioning $u_\mathrm{opt} = \sbvek{u_{\mathrm{opt},1}}{u_{\mathrm{opt},2}}$ corresponding to $\sbvek{f}{x^0}$, cf.~\eqref{eq:graddyn_LQ_state}, and $y_\mathrm{opt} = \sbvek{y_{\mathrm{opt},1}}{y_{\mathrm{opt},2}} = \sbvek{\lambda}{\lambda_0}$ as in \eqref{eq:graddyn_LQ_output}. The interconnection is depicted in Figure~\ref{fig:coupling}.

In this feedback interconnection, the second line in \eqref{eq:coupling} sets the initial value of the underlying optimal control problem by means of the current plant output. Thus, there is no necessity to include an observer. Further, due to $y_{\mathrm{opt},2} = \lambda_0$, and in view of $\lambda(0) = \lambda_0$ for all $(\lambda,\lambda_0)\in D(\calC^*)$ as proven in Lemma~\ref{lem:C}, the last line implements the control value $u_p = -\frac{1}{\alpha}B^\top \lambda(0)$ in the plant. This corresponds to the optimal input at time zero and hence is very closely related to model predictive control, \changed{in which the initial part of the optimal control is fed back into the plant}: Reformulating the second line in the optimality system \eqref{eq:optimality_conditions}, we get
\begin{align*}
    \alpha \hat u + B^\top \hat \lambda = 0
\end{align*}
i.e., the optimal control at the initial time satisfies
\begin{align*}
     \hat u(0)  = -\frac{1}{\alpha}B^\top \hat \lambda(0),
\end{align*}
\changed{which yields a model predictive controller.}
Here, the feedback interconnection depicted in Figure~\ref{fig:coupling} yields a controller
\begin{align*}
    u_p(t) = \changed{-\frac{1}{\alpha}B^\top \lambda_0(t) = -\frac{1}{\alpha}B^\top \lambda(0)(t)}
\end{align*}
corresponding to the initial value of the current (approximate) Lagrange multiplier \changed{as $(\lambda(t),\lambda_0(t))\in D(\mathcal{C}^*)\subset H^1([0,t_f],\R^n)\times \R^n$ is a strong solution due to Proposition~\ref{lem:graddyn_ph}, where $t$ corresponds to the virtual time of the optimizer such that in particular $\lambda(t)(0) = \lambda_0(t)$. However, as this feedback is based on the current suboptimal state of the optimizer dynamics, the interconnection Figure~\ref{fig:coupling} yields a continuous-time suboptimal model predictive control scheme.}

In view of Proposition~\ref{prop:coupling}, the coupling of the optimizer dynamics \eqref{eq:graddyn_LQ} and the plant dynamics \eqref{eq:plant} via \eqref{eq:coupling} leads to the closed-loop monotone pH system
\begin{align}
\label{eq:optimizer_plant_monotone}
    \frac{\mathrm{d}}{\mathrm{d}t}\begin{bmatrix}
         x_p(t)\\
         \sbvek{x(t)}{u(t)}\\
        \sbvek{\lambda(t)}{\lambda_0(t)}
    \end{bmatrix}
    = \begin{bmatrix}
        -M_p(x_p(t))\\ -\nabla J(x(t),u(t))\\0
    \end{bmatrix}+
    \left[\begin{array}{c c c}
        0& 0 &\begin{smallbmatrix}
            0 & -B_pB^\top
        \end{smallbmatrix}\\0& 0& \mathcal{C}^*\\
         \begin{smallbmatrix}
             0\\BB_p^\top 
         \end{smallbmatrix}&        
            -\mathcal{C} & 0
    \end{array}\right]
    \begin{bmatrix}
         x_p(t)\\
         \sbvek{x(t)}{u(t)}\\
        \sbvek{\lambda(t)}{\lambda_0(t)}
    \end{bmatrix}.
\end{align}

\begin{figure}[t]
    \centering
    \includegraphics[width=.8\linewidth]{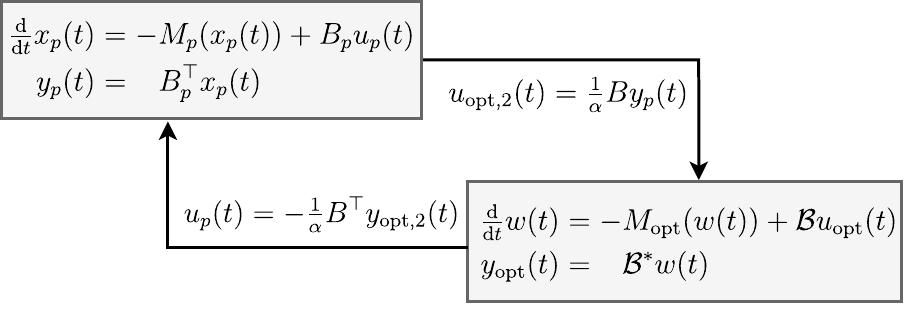}
    \caption{Coupling of plant and optimizer dynamics with optimizer dynamics' state $w=(x,u,\lambda,\lambda_0)$.}
    \label{fig:coupling}
\end{figure}

\noindent The following results show the stability of the interconnected system. \changed{Note that here, we consider stability of the origin, and therefore we consider the stability of the interconnected system in a neighborhood of $\hat x=0$. However, an extension to other setpoints is straightforward.} 
\begin{prop}
\label{prop:coupled}
The coupled system \eqref{eq:optimizer_plant_monotone} is a maximally monotone pH system. Furthermore, if $\nabla J$ is differentiable satisfying \eqref{eq:convex_in_x} with $\hat x=0$ and if the plant fulfills the coercivity assumption 
\[
\langle M_p(x),x\rangle\geq c_p\|x\|^2,
\] then for all initial values $(x_p(0),z(0),\lambda(0),\lambda_0)\in \R^n \times H^1([0,t_f],\R^n)\times L^2([0,t_f],\R^m) \times H^1([0,t_f],\R^n)\times\R^{n}$ the solution of \eqref{eq:optimizer_plant_monotone} satisfies $(x_p,x,u,\lambda,\lambda_0)(t)\rightarrow0$ as $t\rightarrow 0$.
\end{prop}
\begin{proof}
Both subsystems are maximally monotone pH systems, hence, we can apply Proposition~\ref{prop:coupling} to obtain the result that the coupled system \eqref{eq:optimizer_plant_monotone} is a maximally monotone pH system. The convergence follows from   Proposition~\ref{prop:convergence}.
\end{proof}

\noindent We conclude this section with a discussion of the previous stability result. We provide a result, which does not use LaSalle's invariance principle. However, due to the coercivity assumption in the plant model, we required that the uncontrolled plant is already asymptotically stable. Whereas such a property may be ensured with pre-stabilization using a stabilizing feedback controller, relaxing this assumption is subject to future work, in particular in view of the absence of such an assumption in the finite-dimensional case \cite{Pham22}.

\section{Conclusions}
\noindent This paper provides a (nonlinear) monotone port-Hamiltonian system formulation in Hilbert spaces. Besides showing shifted passivity and closedness under power-preserving couplings, we prove asymptotic and local exponential stability of equilibria for monotone pH systems that have a particular block-operator structure that appears, for instance, in optimality conditions for optimal control problems. Furthermore, from these optimality conditions, we derived the continuous-time optimizer dynamics resulting from the primal-dual gradient method. We showed that these dynamics may be viewed as an infinite-dimensional monotone pH system whose solutions converge towards the solution of the optimal control problem. Furthermore, we demonstrated that the optimizer dynamics can be coupled with a monotone plant dynamics in a structure-preserving way and established convergence of the coupled system towards a given steady state.






\bibliographystyle{abbrv}
\bibliography{references.bib}

\begin{appendix}
    \section{Proof of Lemma~\ref{lem:C}}
    \label{sec:proof_lemma_c}
\noindent \underline{\emph{Step 1:}}  We abbreviate $X_{\lambda}:= L^2([0,t_f],\R^n)\times \R^n$ and define
    \begin{align*}
    \mathcal{D} &: L^2([0,t_f],\R^n)\supset D(\calD)  \to X_{\lambda},\\
    D(\calD) &= H^1([0,t_f],\R^n), \quad \calD x := \begin{bmatrix}
        \frac{\mathrm{d}}{\mathrm{d}\tau} x\\
        x(0)
    \end{bmatrix}.
\end{align*}
We now show that $\calD^*: X_{\lambda} \supset D(\mathcal{D}^*) \to L^2([0,t_f],\R^n)$ is given by
\begin{align}
    \calD^*(\lambda,\lambda_0) &= -\frac{\mathrm{d}}{\mathrm{d}\tau} \lambda,\nonumber\\ 
    D(\calD^*) &= \{(\lambda,\lambda_0)\in X_{\lambda}\,|\,\lambda \in H^1([0,t_f],\R^n)\wedge \lambda(0)=\lambda_0\wedge \lambda(t_f)=0\}. \label{eq:domainadj}
\end{align}
To this end, we first consider the definition of the domain of the adjoint operator
\begin{align*}    D(\calD^*) = \{&(\lambda,\lambda_0)\in X_{\lambda}\,|\,\exists z\in L^2([0,t_f],\R^n): \langle \calD x,(\lambda,\lambda_0)\rangle_{X_{\lambda}} = \langle x,z\rangle_{L^2([0,t_f],\R^n)} \, \forall x\in D(\calD)\}.
\end{align*}
We first show that the set on the right-hand side of \eqref{eq:domainadj} is contained in $\dom{\calD^*}$. To this end, let $(v,v_0) \in \{(\lambda,\lambda_0)\in L^2([0,t_f],\R^n)\times \R^n\,|\,\lambda \in H^1([0,t_f],\R^n)\wedge \lambda(0)=\lambda_0\wedge \lambda(t_f)=0\}$. Then, for all $x\in \dom \calD = H^1([0,t_f],\R^n)$, 
\begin{align*}
    \langle \calD x,(v,v_0)\rangle_{X_{\lambda}} &= \int_0^{t_f} \langle \dot x(\tau),v(\tau)\rangle \,\mathrm{d}\tau + \langle x(0),v_0\rangle \\
    &= -\int_0^{t_f} \langle x(\tau),\dot v(\tau)\rangle\,\mathrm{d}\tau + \langle x(t_f),\underbrace{v(t_f)}_{=0}\rangle  - \langle x(0),v(0)\rangle + \langle x(0),v_0\rangle\\
    &= -\int_0^{t_f} \langle x(\tau),\dot v(\tau)\rangle \,\mathrm{d}\tau + \langle x(0),\underbrace{v_0- v(0)}_{=0}\rangle = \langle x,z\rangle, 
\end{align*}
with $z = \dot v \in L^2([0,t_f],\R^n)$ which implies that $(v,v_0)\in \dom{\calD^*}$.\\
Conversely, let $(v,v_0)\in \dom{\calD^*}$. By definition of the adjoint, there is $z\in L^2([0,t_f],\R^n)$ such that
\begin{align}\label{eq:adjointcontained2}
    \langle \calD x,(v,v_0)\rangle_{X_{\lambda}} = \langle x,z\rangle_{L^2([0,t_f],\R^n)} \qquad \forall x\in \dom \calD = H^1([0,t_f],\R^n).
\end{align}
In particular,
\begin{align*}
    \int_0^{t_f} \langle \dot x(\tau),v(\tau)\rangle\,\mathrm{d}\tau = \int_0^{t_f} \langle x(\tau),z(\tau)\rangle \,\mathrm{d}\tau \qquad \forall x \in H^1_0([0,t_f],\R^n)\subset \dom \calD
\end{align*}
which implies that $v\in H^1([0,t_f],\R^n)$ with $\dot v = -z$. Inserting this into \eqref{eq:adjointcontained2} and integration by parts yields
\begin{align*}
    \langle x(t_f),v(t_f)\rangle - \langle x(0),v(0)\rangle + \langle x(0),v_0\rangle = 0 \qquad \forall x\in H^1([0,t_f],\R^n).
\end{align*}
Testing this equation with $x\in H^1([0,t_f],\R^n)$ such that $x(0)=0$ implies $v(t_f)=0$ and conversely, choosing $x\in H^1([0,t_f],\R^n)$ such that $x(t_f)=0$ yields $v(0)=v_0$. Thus, $(v,v_0) \in \{(\lambda,\lambda_0)\in L^2([0,t_f],\R^n)\times \R^n\,|\,\lambda \in H^1([0,t_f],\R^n)\wedge \lambda(0)=\lambda_0\wedge \lambda(t_f)=0\}$ which proves that $\calD^*$ is given by \eqref{eq:domainadj}. 

Since the operator $\calC$ is a bounded additive perturbation of $\calD$, the adjoint of the sum is the sum of the adjoint operators and this proves the claimed formula for $\calC^*$. 
\smallskip

\noindent \underline{\emph{Step 2:}} We show that $\calC$ is surjective. Let $(x^0,f)\in L^2([0,t_f],\R^n)\times \R^n$ and let $x\in D(\mathcal{D}) = H^1([0,t_f],\R^n)$ be the solution of $\dot x = Ax + f$, $x(0)=0$. Further, set $u=0$ such that $(x,u) \in \dom{\mathcal{C}}$. Then, $\mathcal{C}\sbvek{x}{u} = \sbvek{f}{x^0}$. Hence, $\mathcal{C}$ is surjective and in particular it has closed range. Thus, from the closed range theorem, we conclude $\ker (\mathcal{C}^*) = \ran (\mathcal{C})^\perp = \{0\}$.
\end{appendix}

\end{document}